\documentclass[11pt,A4paper]{article}

\usepackage[left=30mm,right=30mm,top=25mm,bottom=30mm]{geometry}
\usepackage{labelfig}
\usepackage{epsfig}
\usepackage{epstopdf}
\usepackage{soul}
\usepackage{xfrac}
\usepackage[T1]{fontenc}

\usepackage[percent]{overpic}

\usepackage{color}
\usepackage{amsthm,amsmath,amssymb}
\usepackage{booktabs}
\usepackage{mathpazo}
\usepackage{microtype}
\usepackage{overpic}
\usepackage{bm}
\usepackage{sectsty}
\usepackage[
	pdftitle={PDFTitle},
	pdfauthor={Hugo Parlier},
	ocgcolorlinks,
	linkcolor=linkred,
	citecolor=linkred,
	urlcolor=linkblue]
{hyperref}

\definecolor{linkred}{RGB}{0,191,255} %DeepSkyBlue
\definecolor{linkblue}{RGB}{16, 78, 139}

\usepackage[hang,flushmargin]{footmisc}
\usepackage{enumitem}
\usepackage{titlesec}
	\titlespacing{\section}{0pt}{12pt}{0pt}
	\titlespacing{\subsection}{0pt}{6pt}{0pt}
	
\titlelabel{\thetitle.\quad}

\makeatletter 

\long\def\@footnotetext#1{% 
\H@@footnotetext{% 
\ifHy@nesting 
\hyper@@anchor{\@currentHref}{#1}% 
\else 
\Hy@raisedlink{\hyper@@anchor{\@currentHref}{\relax}}#1% 
\fi 
}}

\def\@footnotemark{% 
\leavevmode 
\ifhmode\edef\@x@sf{\the\spacefactor}\nobreak\fi 
\H@refstepcounter{Hfootnote}% 
\hyper@makecurrent{Hfootnote}% 
\hyper@linkstart{link}{\@currentHref}% 
\@makefnmark 
\hyper@linkend 
\ifhmode\spacefactor\@x@sf\fi 
\relax 
}% 

\makeatother 

\theoremstyle{plain}
\newtheorem{theorem}{Theorem}[section]
\newtheorem*{theorem-otal}{Theorem 1.3}
\newtheorem{proposition}[theorem]{Proposition}
\newtheorem{lemma}[theorem]{Lemma}
\newtheorem{corollary}[theorem]{Corollary}

\theoremstyle{definition}

\newtheorem{remark}[theorem]{Remark}

\newcommand{\area}{{\rm area}}

\newcommand{\Cr}{{\rm cr}}

\sectionfont{\large \bfseries}
\subsectionfont{\normalsize}

\long\def\symbolfootnote[#1]#2{\begingroup%
\def\thefootnote{\fnsymbol{footnote}}\footnote[#1]{#2}\endgroup}

\def\blfootnote{\xdef\@thefnmark{}\@footnotetext}

\begin{document}

{\Large \bfseries Crossing number inequalities for curves on surfaces}

{\large Alfredo Hubard and Hugo Parlier\symbolfootnote[0]{\small 
}}

\vspace{0.5cm}

{\bf Abstract.}
We prove that, as $m$ grows, any family of $m$ homotopically distinct closed curves on a surface induces a number of crossings that grows at least like $(m \log m)^2$. We use this to answer two questions of Pach, Tardos and Toth related to crossing numbers of drawings of multigraphs where edges are required to be non-homotopic. Furthermore, we generalize these results, obtaining effective bounds with optimal growth rates on every orientable surface.
\vspace{0.5cm}

\section{Introduction} \label{s:introduction}

Due to its beauty and large number of applications, the so-called crossing lemma for planar drawings of simple graphs has become a corner stone in the theory of topological graphs. Among multiple generalizations, the inequality was generalized to drawings of graphs (sometimes called multigraphs) where edges are required to be pairwise non-homotopic. Thought of in topological terms, this is equivalent to representing edges of the graph by arcs, and vertices by marked points of the plane or on a sphere. 

Pach, Tardos and Toth show that any such drawing with $m$ edges and $n$ vertices, for $m>4n$, must have at least, on the order of, $m^2/n$ crossings. 
They also show that for fixed $n$, the crossing number is super-quadratic in $m$ and provide a family of examples in which the order of growth is $\sfrac{m^2}{n} \log^2 \sfrac{m}{n}$.

We show that their upper bound does indeed have the correct order of growth, and simultaneously generalize it to drawings on surfaces of any genus $g\geq 0$. The Euler characteristic $\chi = 2-2g-n$ of the surface with the vertices removed is a useful quantity in this context:

\begin{theorem}\label{thm:graphs}
Let $G$ be a graph with $n$ vertices and $m$ edges drawn on a closed surface of genus $g\geq 0$ such that no two edges are homotopic. Then, for  $m \geq e^6(|\chi|+1)$, its crossing number satisfies
\[\Cr(G)\geq  \frac{1}{512|\chi|} m^2 \left({\log^2 \frac{m}{e^6(|\chi|+1)} - 256 |\chi| }\right).
\]
\end{theorem}

To show this, we in fact show a crossing result for curves on surfaces.
\begin{theorem}\label{thm:curves}
    Let $\Gamma$ be a collection of $m$ distinct homotopy classes of non-trivial closed and primitive curves on a surface $\Sigma$ of Euler characteristic $\chi$. Then, for all $m \geq  \,e^6(|\chi|+1)$, we have
    $$
    \Cr(\Gamma) > \frac{1}{128 |\chi|} \left(m \log\left(\frac{m}{(|\chi|+1) e^6}\right)\right)^2.
    $$
\end{theorem}
The crossing $\Cr(\Gamma)$ is really just another term for intersection, but we use the term crossing in analogy with the graph setting. And indeed Theorem \ref{thm:graphs} above follows from this result via a concatenation observation (see Lemma \ref{lem:cutandpaste}). There have also been other recent results about crossing numbers for curves \cite{baader2024,parlier-hubard,jorg2024} but, unlike here, they mostly concern simple curves. 

The proof of Theorem \ref{thm:curves} uses the geometry of hyperbolic surfaces in an essential way. We use the Koebe-Andreev-Thurston uniformization theorem to associate to a collection of curves a hyperbolic surface (as in, for instance, \cite{AougabSouto2018}), and in a way that is not dissimilar to the separator theorem \cite{miller-thurston, spielman}. A crucial ingredient in our setting is the use of length estimates on the first $m$ primitive closed geodesics on a hyperbolic surface of given topology (Proposition \ref{prop:average}).

Pach, Tardos and Toth \cite{pach-tardos-toth}, in their approach to a crossing lemma for graph drawings in the plane with non-homotopic edges, find upper bounds on the size of a family of homotopically distinct arcs which both self-intersect at most $k$ times and pairwise intersect at most $k$-times. They also provided lower bounds on this problem by constructing families of arcs  \cite{pach-tardos-toth} of size $2^{\sqrt{nk}}$ and asked the question about matching upper bounds. In both geometric topology and combinatorics, there is a plethora of results related to these types of problems (see for instance \cite{ABG2019,valtr,fox-pach-suk,Greene2019,mohar, pach-toth, Prz2015, SmithPrz2019}). In answer to this question, in \cite{wood} it was shown that such a family has cardinality at most $6^{13nk}$. While up until now, these upper bounds were used to prove crossing lemmas, we take an opposite approach and use our crossing number result to prove upper bounds on the number of curves and arcs. For curves we get:

\begin{corollary}\label{cor:curves}
Let $k\geq 1$. If $\Gamma$ is family of $m$ homotopically distinct closed curves on $\Sigma$, such that $i(\alpha,\beta)\leq k$ for all $\alpha,\beta\in \Gamma$, then 
$$m\leq e^{6\sqrt{ 2 k |\chi|} +\log(|\chi|+1) +6 }.$$
\end{corollary}

For arcs we obtain:

\begin{corollary}\label{cor:arcs}
Let $k\geq 1$. If $\mathcal{A}$ is family of $m$ distinct arcs on $\Sigma$, such that $i(a,b)\leq k$ for all $a,b\in \Gamma$, then 
$$m\leq e^{24\sqrt{ (k+1) |\chi|} +\log(|\chi| +1)+6 }.$$
\end{corollary}

When restricted to arcs drawn in the plane, these bounds match, up to the constant in the exponential, the lower bounds in \cite{pach-tardos-toth}.

In the final section (Section \ref{sec:optimality}) of our paper, we construct examples of arcs and curves which match, up to constants, our lower bounds for crossing estimates (Theorems \ref{thm:graphcrossingexample} and \ref{thm:curvesexample}). The examples we construct, while inspired by hyperbolic surfaces, are in fact combinatorial. The combinatorial setup, as in the work of Chas or Despr\'e-Lazarus \cite{Chas, Despre-Lazarus}, is very useful when quantifying intersection properties of curves and arcs.

\noindent{\bf Acknowledgements.}

We are grateful to Arnaud de Mesmay for useful comments and encouragements, and to Jasmin J\"org and Sebastian Baader for inspiring conversations on crossing numbers.

\section{Preliminaries}

We briefly review some facts about arcs and curves on surfaces and some basic hyperbolic geometric tools useful for their study. 
Throughout $\Sigma$ will be a topological finite-type orientable surface. It is determined by its genus $g$ and number of punctures $n$. As we will be drawing arcs and graphs on $\Sigma$, it will be convenient to think of the punctures as marked points, so we fill in the punctures so that they can serve as endpoints of arcs (or, equivalently, drawings of edges of graphs) to obtain a surface $\overline{\Sigma} = \Sigma \cup P$ where $P$ is the set of marked points. Its Euler characteristic is $\chi = \chi (\Sigma) = 2-2g -n$ and we will only be considering $\Sigma$ such that $\chi <0$. For instance, $\Sigma$ can be a thrice punctured sphere or a genus $2$ closed surface. 

A closed curve is the continuous image of a circle on $\Sigma$ and we consider them up to (free) homotopy. In particular, we are only interested in closed curves that are not homotopic to a point (or a puncture / marked point). We only consider primitive closed curves, that is curves that are not iterates of another curve. Note that $|\chi|$ is the number of three-holed spheres on the complement of maximal set of disjoint and homotopically distinct simple closed curves (where maximality is with respect to inclusion). The three-holed spheres are generally called pants or pairs of pants. A maximal set of simple closed curves is of size $3g-3+n$ and is called a pants decomposition

An arc is the continuous image of a closed interval on $\overline{\Sigma}$ whose endpoints lie on $P$ and such that the image of the corresponding open interval lies entirely in $\Sigma$. Again, we consider arcs up to homotopy, but this time such that the homotopy fixes the endpoints and such that, outside of the endpoints, the homotopies are not allowed to pass through the marked points $P$. Again we exclude trivial arcs (which can only arise if an arc has both endpoints on a single marked point). Note that, due to the fact that they have endpoints, arcs are always primitive. 

We define intersection as an invariant of homotopy classes. Namely, for curves
$$
i(\alpha,\beta) = \min_{\alpha' \sim \alpha,\beta'\sim \beta} |\alpha' \cap \beta'|
$$
where $\sim$ means freely homotopic and $| \cdot |$ denotes cardinality, and we assume that there are no triple intersection points between $\alpha'$ and $\beta'$. Alternatively, $i(\alpha,\beta)$ is the minimum number of transversal intersections among representatives of the respective homotopy classes.

When you apply the definition above to the intersection of a curve with itself, this always gives an even number. Hence it is standard to define the self-intersection number of a curve $\alpha$ as being the quantity $\frac{1}{2} i(\alpha,\alpha)$. This is its minimal number of transversal self-intersections for any representation without triple points. For arcs, we define intersection and self-intersection analogously. That is
$$
i(a,b) = \min_{a' \sim a,b'\sim b} |a' \cap b'|
$$
but with $|a' \cap b'|$ denoting cardinality outside of their endpoints and again under the assumption that there are no triple intersection points. An arc or a curve is said to be simple if it has no self-intersection points. When we have a collection $\mathcal{C}$ of curves or arcs, by its crossing $\Cr(\mathcal{C})$ we mean the sum of all the intersections. 

Since $\Sigma$ satisfies $\chi(\Sigma)<0$, it admits hyperbolic metrics, that is complete finite area hyperbolic metrics $X$, such that $X$ and $\Sigma$ are homeomorphic (denoted $X\cong \Sigma$). The finite area/complete metric condition forces the punctures to be realized by cusps. In this context, given a hyperbolic surface $X \cong \Sigma$, any closed curve is realized by a unique closed geodesic. Similarly, an arc is realized by a unique bi-infinite geodesic between the corresponding cusps. 

A highly used and convenient fact about geodesic representatives of homotopy classes of closed curves and arcs on hyperbolic surfaces is that they minimize intersection. More precisely, they minimize both self and pairwise intersection (but you have to count intersection with multiplicity in the event that the hyperbolic metric imposes a triple intersection point). Given a hyperbolic surface, we denote by $\ell_X(\gamma)$ the length of the unique closed geodesic in the homotopy class of a curve $\gamma$. 

On a given hyperbolic surface $X$ of genus $g$ with $n$ cusps, we can bound the number of closed geodesics of given length. This is essentially Lemma 6.6.4 from \cite{BuserBook}, but the statement is slightly modified in the presence of cusps.

\begin{lemma}\label{lem:buser}
Let $X$ be a hyperbolic surface of genus $g$ with $n$ cusps. The number of primitive closed geodesics of $X$ of length at most $L$ is bounded above by 
$$
(g+  \lceil n/2 \rceil -1 )\, e^{L+6}
$$
\end{lemma}

\begin{proof}
If $X$ is closed, we have $n=0$ and the statement is exactly Lemma 6.6.4 from \cite{BuserBook}. If $n$ is even, by pairing the cusps together, one can think of it as a degenerate closed surface of genus $g+n/2$ where $n/2$ curves have been pinched to become length $0$. These are called noded surfaces, see for instance \cite{BuserBook}. If $n$ is odd, by adding a once-cusped torus, we can obtain a noded surface of genus $g + \lceil n/2 \rceil$. In either case, we can apply Lemma 6.6.4 from \cite{BuserBook} to obtain the result.
\end{proof}

We will use this result to show, that given $m$ curves on a hyperbolic surface, as $m$ grows, the average length of these curves must be on the order of $\log(m)$, and this independently of the choice of hyperbolic metric $X$.

\begin{proposition}\label{prop:average}
Let $\gamma_1,\hdots,\gamma_m$ be a collection of distinct closed and primitive closed curves on $\Sigma$ and $\lambda>0$. Then for any hyperbolic metric $X \cong \Sigma$ we have 
$$
\sum_{k=1}^{m} \ell_X(\gamma_k) > \frac{\lambda}{1+\lambda}\left(  \log(m) -  \left(\log((1+\lambda)(g+ \lceil n/2 \rceil-1)+6\right)\right).
$$
\end{proposition}

\begin{proof}
For $\lambda >0$, we set $L=L(\lambda,m)$ to be such that
$$
\frac{1}{1+\lambda} m = (g+ \lceil n/2 \rceil-1) e^{L+6}
$$
that is
$$
L= \log(m) - (\log((1+\lambda)(g+ \lceil n/2 \rceil-1)+6).
$$
Now since $m= \left(\frac{\lambda}{1+\lambda} + \frac{1}{1+\lambda}\right) m$,  by Lemma \ref{lem:buser}, at least $\frac{\lambda}{1+\lambda} m$ curves, are of length strictly greater than $L$, so in particular 
$$
\sum_{k=1}^{m} \ell_X(\gamma_k) > \frac{\lambda}{1+\lambda} m L = \frac{\lambda}{1+\lambda}m \left(  \log(m) -  \left(\log((1+\lambda)(g+ \lceil n/2 \rceil-1)+6\right)\right).
$$
\end{proof}
Note that the inequality only has any content provided $m$ is large enough, namely under the condition that 
$$
\log(m) > \log((1+\lambda)(g+ \lceil n/2 \rceil)-1) +6
$$
or, stated differently, that
$$
m> (1+\lambda)(g+ \lceil n/2 \rceil-1))\, e^6.
$$
To simplify our computations, in the proof of Theorem \ref{thm:curves}, we will make use of this proposition for $\lambda=1$. 

\section{The proof of the crossing theorem for curves}

In this section we prove Theorem \ref{thm:curves} from the introduction. 

Let $\Gamma$ be a collection of $m$ primitive and non-pairwise homotopic curves of $\Sigma$, all represented in minimal position and without any triple intersection points. Its intersection number is $\Cr(\Gamma)$. We think of $\Gamma$ as a graph drawing on $\Sigma$ with vertices the set of intersection and self-intersection points. By adding simple (and non-trivial) arcs, we can complete $\Gamma$ into a triangulation $T$, that is a decomposition of $\Sigma$ into triangles, without increasing the number of vertices. Note that $T$ is not necessarily a simplicial triangulation, and in particular it can contain (non-trivial) loops, and all of its edges are non-trivial simple arcs.

We now appeal to the circle packing theorem for triangulations embedded on surfaces (see \cite{He-Schramm} or \cite{Stephenson}). It states that there exists a unique hyperbolic metric $X \cong \Sigma$ and a circle packing on $X$ such that there exists a homeomorphism from $\Sigma$ to $X$ which sends the topological embedding of $T$ to the intersection graph of the circle packing, drawn by connecting the centers of touching circles. This gives us a geometric representative of $T$ with the length of an edge equal to the sum of the two radii of the touching circles. We denote by $r_i$, for $i=1,\hdots, \Cr(\Gamma)$, the radii of the circles (or disks). 

We can realize $\Gamma$ as a collection of closed geodesics on $X$. As closed geodesics are of minimal length in their homotopy class, they are of length at most the length of their geometric representatives. In particular
\begin{equation}\label{eq:length}
\ell_X(\Gamma) \leq 4 \sum_{i=1}^{\Cr(\Gamma)} r_i.
\end{equation}
as exactly two curves of $\Gamma$ pass through each intersection point and, on the corresponding disk, their length is twice the radius. Now we can apply the Cauchy-Schwartz inequality to obtain 
$$
 \sum_{i=1}^{\Cr(\Gamma)} r_i \leq \sqrt{\Cr(\Gamma)} \sqrt{\sum_{i=1}^{\Cr(\Gamma)} r_i^2}.
$$
Now if $D_i$ is the disk of radius $r_i$, as it is hyperbolic, its area is always greater than the area of a Euclidean disk of radius $r_i$. Furthermore, the disks are all disjoint and thus of total area less than the area of $X$. Hence the above inequality gives
\begin{equation}\label{eq:area}
 \sum_{i=1}^{\Cr(\Gamma)}r_i \leq \frac{1}{\sqrt{\pi}} \sqrt{\Cr(\Gamma)} \sqrt{\sum_{i=1}^{\Cr(\Gamma)} \pi r_i^2} < \frac{1}{\sqrt{\pi}} \sqrt{\Cr(\Gamma)} \sqrt{\area(X)}=  \frac{1}{\sqrt{\pi}} \sqrt{\Cr(\Gamma)} \sqrt{2 \pi |\chi|}
\end{equation}
Putting together Equations \ref{eq:length} and \ref{eq:area} we obtain
$$
\Cr(\Gamma) > \frac{1}{32 |\chi|} \left( \ell_X(\Gamma) \right)^2.
$$
Now we can apply Proposition \ref{prop:average} with $\lambda=1$ which tells us that 
\begin{eqnarray*}
 \ell_X(\Gamma) & >& \frac{1}{2}m \left(\log(m) - \left(\log(2(g+ \lceil n/2 \rceil-1)+6\right)\right)\\
 &=& \frac{1}{2} m \log\left(\frac{m}{2(g+\lceil n/ 2 \rceil-1) e^6}\right)
 \end{eqnarray*}
 In order to ensure that the right-hand quantity is positive, we now require that 
 \begin{eqnarray*}
\log\left(\frac{m}{2(g+\lceil n/ 2 \rceil-1) e^6}\right) & \geq & 0  \,\,\,\,\mbox{and so}\\
 m &\geq & 2\,e^6(g+\lceil n/ 2 \rceil-1).
 \end{eqnarray*}
Taking into account the parity of $n$, in order to ensure a clean condition on $m$ that only depends on the Euler characteristic $|\chi | = 2g-2 +n$
we will thus require that 
$$
 m \geq  e^6(|\chi|+1).
$$
We can conclude:
\begin{eqnarray*}
\Cr(\Gamma) & >& \frac{1}{128 |\chi|} \left(m \log\left(\frac{m}{2(g+\lceil n/ 2 \rceil-1) e^6}\right)\right)^2\\
& \geq & \frac{1}{128 |\chi|} \left(m \log\left(\frac{m}{(|\chi|+1) e^6}\right)\right)^2
 \end{eqnarray*}
as required.

\section{Two corollaries: bounding sets of curves and a crossing theorem for graphs}
\subsection{From curves to arcs}
Recall the result for graph drawings that we want to prove.
\begin{theorem}[Theorem \ref{thm:graphs}]
Let $G$ be a graph with $n$ vertices and $m$ edges drawn on a closed surface of genus $g\geq 0$ such that no two edges are homotopic. Then, for  $m \geq e^6(|\chi|+1)$, its crossing number satisfies
\[\Cr(G)\geq  \frac{1}{512|\chi|} m^2 \left({\log^2 \frac{m}{e^6(|\chi|+1)} - 256 |\chi| }\right).
\]
\end{theorem}

This statement is, verbatim, the statement that a set $\mathcal{A}$ of arcs of size $m$ has crossing number $\Cr(\mathcal{A})$ that satisfies the same inequality above. Note that the statement only has content provided the right side of the inequality is positive.

To deduce this statement from our theorem about curves, we will need to associate a closed curve to an arc. We do so as follows: take an arc $a$ between marked points $p$ and $q$ (which are not necessarily distinct). We orient $a$ as going from $p$ to $q$. We, topologically, can think of $p$ and $q$ being small loops, say $\alpha_p$ and $\alpha_q$, oriented in the positive direction. We obtain a closed curve $\gamma_a$ by taking $a*\alpha_q*a^{-1}*\alpha_p$ where $*$ represents concatenation. We don't care about the orientation of $\gamma_a$, so similarly, $\alpha_a$ can be thought of as taking the boundary of a small neighborhood of $a$ (see Figure \ref{fig:arc2curve}).

\begin{figure}[h]
%\ShowGrid
{\color{linkblue}
\leavevmode \SetLabels
\L(0.5*.83) $a$\\
\L(0.63*.33) $\gamma_a$\\
\L(0.25*.37) $\alpha_p$\\
\L(0.73*.10) $\alpha_q$\\
\L(0.33*.49) $a$\\
\L(0.5*.20
) $a^{-1}$\\
\endSetLabels
\begin{center}
\AffixLabels{\centerline{\epsfig{file =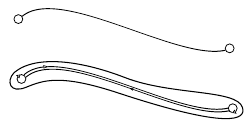,width=8.0cm,angle=0} }}
\vspace{-30pt}
\end{center}
\caption{From an arc to a curve} \label{fig:arc2curve}
}
\end{figure}
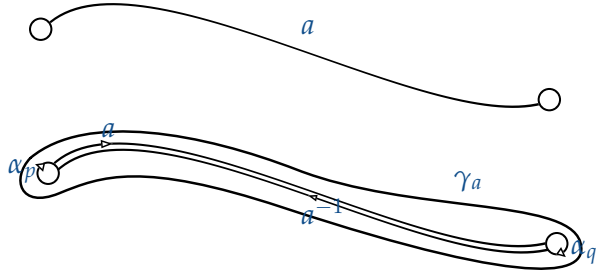

Notice that if $a \neq b$, then $\gamma_a \neq \gamma_b$. Furthermore, if $a$ is a non-trivial arc, then $\gamma_a$ is non-trivial and primitive. 

\begin{lemma}\label{lem:cutandpaste}
Let $a$, $b$ be arcs on a surface, and let $\gamma_a$ and $\gamma_b$ be the unique closed curves described above. Then
$$
i(\gamma_a,\gamma_b) \leq 4 i(a,b) + 4.
$$
\end{lemma}

\begin{proof}
An intersection point between $a$ and $b$ induces $4$ intersection points between $\gamma_a$ and $\gamma_b$ (see Figure \ref{fig:arc2curve2}). 

\begin{figure}[h]
%\ShowGrid
{\color{linkblue}
\leavevmode \SetLabels
\endSetLabels
\begin{center}
\AffixLabels{\centerline{\epsfig{file =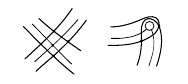,width=8.0cm,angle=0} }}
\vspace{-30pt}
\end{center}
\caption{Intersections} \label{fig:arc2curve2}
}
\end{figure}

The additional $+4$ comes from the endpoints of the arcs: notice that if two arcs have an endpoint in common, this necessarily creates $2$ intersection points between the corresponding curves. As arcs have 2 endpoints, this results in (at most) 4 extra intersection points. 
\end{proof}

We now prove Theorem \ref{thm:graphs}. Consider $\mathcal{A}$. a collection of $m$ arcs on $\Sigma$. We transform each of the arcs into a closed curve. If two curves have intersection $k$, the corresponding arcs have intersection at least $\frac{k}{4}-1$. 

We now apply Theorem \ref{thm:curves} to this set $\Gamma$ of curves. It has intersection at least 

$$
 \frac{1}{128 |\chi|} \left(m \log\left(\frac{m}{e^6(|\chi|+1) }\right)\right)^2
 $$
and so $\mathcal{A}$ must have total intersection at least $\frac{1}{4}$ that intersection, from which we substract at most $\frac{m^2}{2}$ (one for each pair of distinct arcs and 1/2 for each arc). The result follows from a simple computation.

\subsection{Curves and arcs which pairwise intersect at most $k$ times}

We begin by considering sets of curves with self and pairwise intersection bounded above by an integer $k$. 

\begin{corollary}[Corollary \ref{cor:curves}]
Let $k\geq 1$. If $\Gamma$ is family of $m$ homotopically distinct closed curves on $\Sigma$, such that $i(\alpha,\beta)\leq k$ for all $\alpha,\beta\in \Gamma$, then 
$$m\leq e^{6\sqrt{ 2 k |\chi|} +\log(|\chi|+1) +6 }.$$
\end{corollary}

\begin{proof}
We can suppose that $m$, the size of $\Gamma$, satisfies $m > 8$, otherwise the estimate holds trivially.

Note that while there are $m^2$ ordered pairs of curves in $\Gamma$, there are $m(m-1)/2 +m = m(m+1)/2$ unordered pairs of curves of $\Gamma$ (we allow pairs of identical curves). In total, by Theorem \ref{thm:curves}, these create at least 
$$
\Cr(\Gamma) > \frac{1}{128 |\chi|}\left( m \log \left(\frac{m}{e^6 (|\chi|+1)}\right) \right)^2=f(m)
$$
intersection points. In particular, there is a pair $\alpha,\beta \in \Gamma$ that satisfies
$$
i(\alpha,\beta) > \frac{f(m)}{m(m+1)/2}> \frac{1}{64 |\chi |} \frac{m}{m+1} \log^2\left(\frac{m}{e^6 (|\chi|+1)}\right)>\frac{1}{72 |\chi |} \log^2\left(\frac{m}{e^6( |\chi|+1)}\right)
$$
where we just use $m>8$ for the last inequality.

By hypothesis, $i(\alpha,\beta) \leq k$, and hence 
$$
 \frac{1}{72 |\chi|} \log^2\left(\frac{m}{e^6 (|\chi|+1)}\right)<k
$$
and so 
$$
m < e^6( |\chi| +1)e^{6\sqrt{2|\chi| k}}
$$
as claimed.
\end{proof}

Now we pass to the case of arcs, again using Lemma \ref{lem:cutandpaste}.

\begin{corollary}[Corollary \ref{cor:arcs}]
Let $k\geq 1$. If $\mathcal{A}$ is family of $m$ distinct arcs on $\Sigma$, such that $i(a,b)\leq k$ for all $a,b\in \Gamma$, then 
$$m\leq e^{24\sqrt{ (k+1) |\chi|} +\log(|\chi| +1)+6 }.$$
\end{corollary}
\begin{proof}
By Lemma \ref{lem:cutandpaste}, if any two arcs $a,b$ have intersection at most $k$, then the corresponding curves $\gamma_a$ and $\gamma_b$ have intersection at most $4k+4$. It suffices to replace $k$ by $4k+4$ in the formula of Corollary \ref{cor:curves}.
\end{proof}

\section{Optimality of the bounds}\label{sec:optimality}
We end with a construction showing that our bounds on the crossing number for collections of curves are roughly optimal. 

The basic idea for the strategy comes from the study of hyperbolic surfaces. Take a hyperbolic surface with injectivity radius bounded from below by a positive constant. We then observe that any two geodesics of length $\ell$ cannot intersect more than roughly $\ell^2$ times (this technique is used in \cite{Erlandsson-Parlier} quantifying an approach by Basmajian \cite{Basmajian} on the relationship between length and intersection). Now we consider all closed geodesics of length up to a given $L$ and we let $L$ grow. As the surface is hyperbolic, there are exponentially many of them, so we have roughly $m=a^L$ of them and so we get an upper bound on their length of roughly $\log(m)$. This gives us a family of curves of size $m$ with at most $m^2 \left(\log(m)\right)^2$ crossings. 

We now implement this strategy in a combinatorial setting, with the idea of getting explicit bounds. In particular, we are not going to be too picky, and will favor the simplicity of the arguments over the optimality of the constants obtained. 

Before explaining our constructions, we give away the punchlines. 

For arcs, our construction will yield the following example:

\begin{theorem}\label{thm:graphcrossingexample}
For any $n\geq 2$ and provided $m\geq n-2$, there exists a non-homotopic multigraph with $n$ vertices and $m$ with total crossings at most
$$
\frac{m^2}{n-2} \left(\log_2\left(\frac{8m}{n-2}\right)\right)^2.
$$
\end{theorem}

For curves, our construction results in the following:

\begin{theorem}\label{thm:curvesexample}
There exists a collection of $m$ curves on $\Sigma$ with total crossing number at most 
$$
\frac{m^2}{|\chi|}\left( \log_2\left( \frac{4m}{3|\chi|}\right)\right)^2.
$$
\end{theorem}

The examples are all based on a construction on a pair of pants which we now explain. 

\subsection{Curves and arcs on a combinatorial pair of pants}

Instead of working with a hyperbolic surface, in the interest of obtaining explicit bounds, we do a combinatorial analogue. To begin with, we construct a combinatorial pair of pants out of Euclidean squares. To do this pair six squares of side lengths $\frac{1}{2}$ as in Figure \ref{fig:SquarePants} to obtain a right-angled hexagon with side lengths $1$ and a singular point in the middle (with total angle $3\pi$). We then double the hexagon to obtain a pair of pants as in \ref{fig:SquarePants} with three cuff lengths of length $2$. Note that the length of the shortest curve not homotopic to a point is $2$ (and this length is realized by the three cuff curves, which we call its systole). We denote by $P$ this pair of pants. It has a frontside, consisting in the original hexagon, and a backside, made of its double.

Note that the surface is locally Euclidean, except in its two singular points with $3\pi$ angle. Hence it is non-postively curved, and there is a notion of being geodesic, meaning locally distance minimizing.

\begin{figure}[h]
%\ShowGrid
{\color{linkblue}
\leavevmode \SetLabels
\endSetLabels
\begin{center}
\AffixLabels{\centerline{\epsfig{file =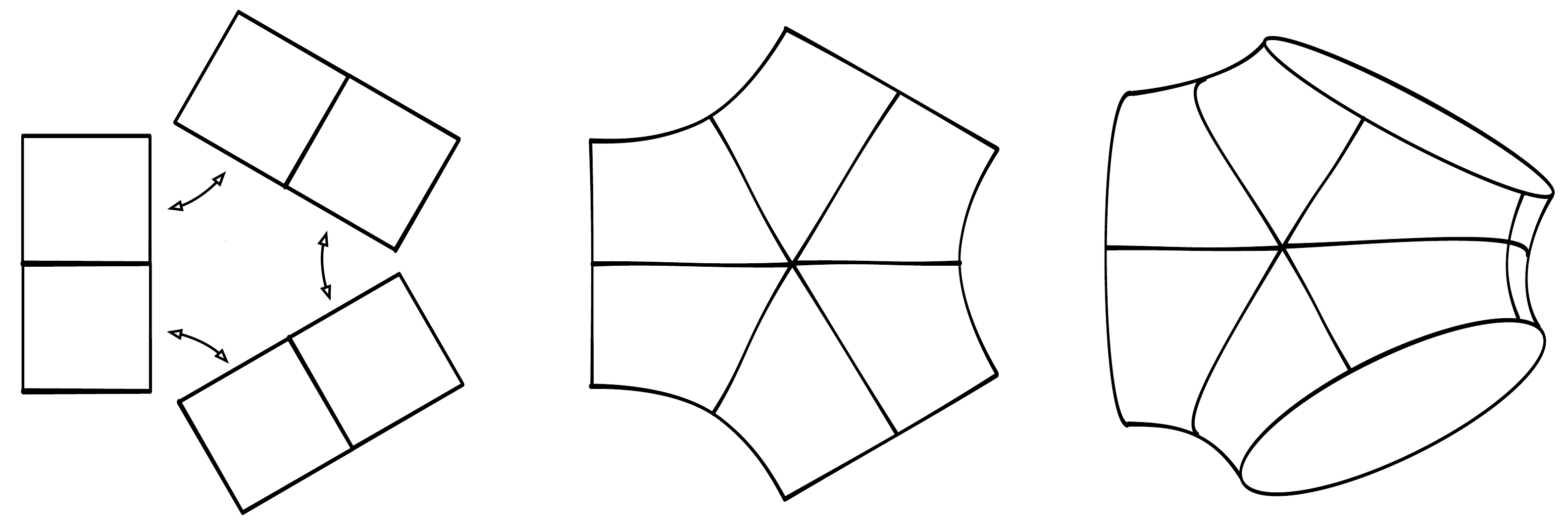,width=12.0cm,angle=0} }}
\vspace{-30pt}
\end{center}
\caption{Constructing a pair of pants $P$ with $12$ squares} \label{fig:SquarePants}
}
\end{figure}

On $P$ we will be considering homotopy classes of arcs with endpoints on the cuffs, and closed curves. As we will see, these both have a nice combinatorial description and natural geodesic representatives. A geodesic representative of a free homotopy class is a global length minimizer. In light of the non-positive curvature, any representative that is locally minimizing everywhere will be globally minimizing.

\begin{remark}While we won't need this explicitly, it is interesting to note that the geodesic representatives get "stuck" on the singular points. More precisely with the exception of the shortest closed curves (the cuffs) and the shortest non-trivial arcs, all geodesic representatives necessarily pass through the singular points. They lie entirely on the orthogonal interior paths marked on Figure \ref{fig:SquarePants}. Independently of this, we will make use of the fact that each homotopy class {\it can} be represented by such a path.
\end{remark}

The main feature we want to use for geodesic minimizers is that, like hyperbolic surfaces, they realize minimal transversal intersection. In light of the lower bound of $2$ on the systole, any two different segments of length $1$ can transversally intersect at most $1$ time. Otherwise they either bound a bigon (and hence do not intersect minimally - see for instance \cite{primer}) or they form a non-trivial closed curve of length strictly less than $2$, which is impossible. This gives a quadratic upper bound on their intersection. We write this observation as a lemma:

\begin{lemma} Let $a,b$ be two geodesic arcs on $P$. Then
$$
i(a,b) \leq \lceil \ell(a)  \rceil*\lceil \ell(b)  \rceil
$$
Similarly, let $\alpha$ and $\beta$ be two geodesic closed curves on $P$. Then 
$$
i(\alpha,\beta) \leq \lceil \ell(\alpha) \rceil*\lceil \ell(\beta) \rceil.
$$
\end{lemma}

\begin{proof}
The proofs are identical in both cases: we break the arcs or curves into segments of length $1$. The ceiling function is in case there is any length left, but in our case all geodesic representatives have integer length. 

By the above observation, any two segments can intersect at most once. The result follows.
\end{proof}

\subsection{Counting curves and arcs on a pair of pants}

We now count arcs and curves of given length on the pair of pants $P$. 

\subsubsection{Arcs}
We begin by counting oriented arcs. Notice that each such arc begins on one of our boundary curves, hence begins by a segment of length $\frac{1}{2}$. We label these initial  segments $f_1,f_2,f3$ (on the front of the pair of pants) and $b_1,b_2,b_3$ (on the back of the pair of pants). This is illustrated in Figure \ref{fig:arcs}.

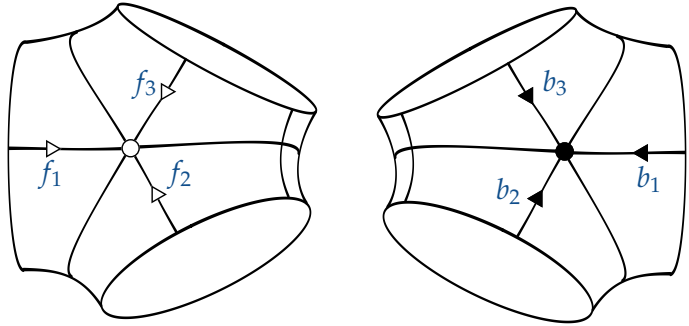
\begin{figure}[h]
%\ShowGrid
{\color{linkblue}
\leavevmode \SetLabels
\L(0.23*.45) $f_1$\\
\L(0.34*.45) $f_2$\\
\L(0.31*.7) $f_3$\\
\L(0.74*.435) $b_1$\\
\L(0.62*.4) $b_2$\\
\L(0.66*.7) $b_3$\\
\endSetLabels
\begin{center}
\AffixLabels{\centerline{\epsfig{file =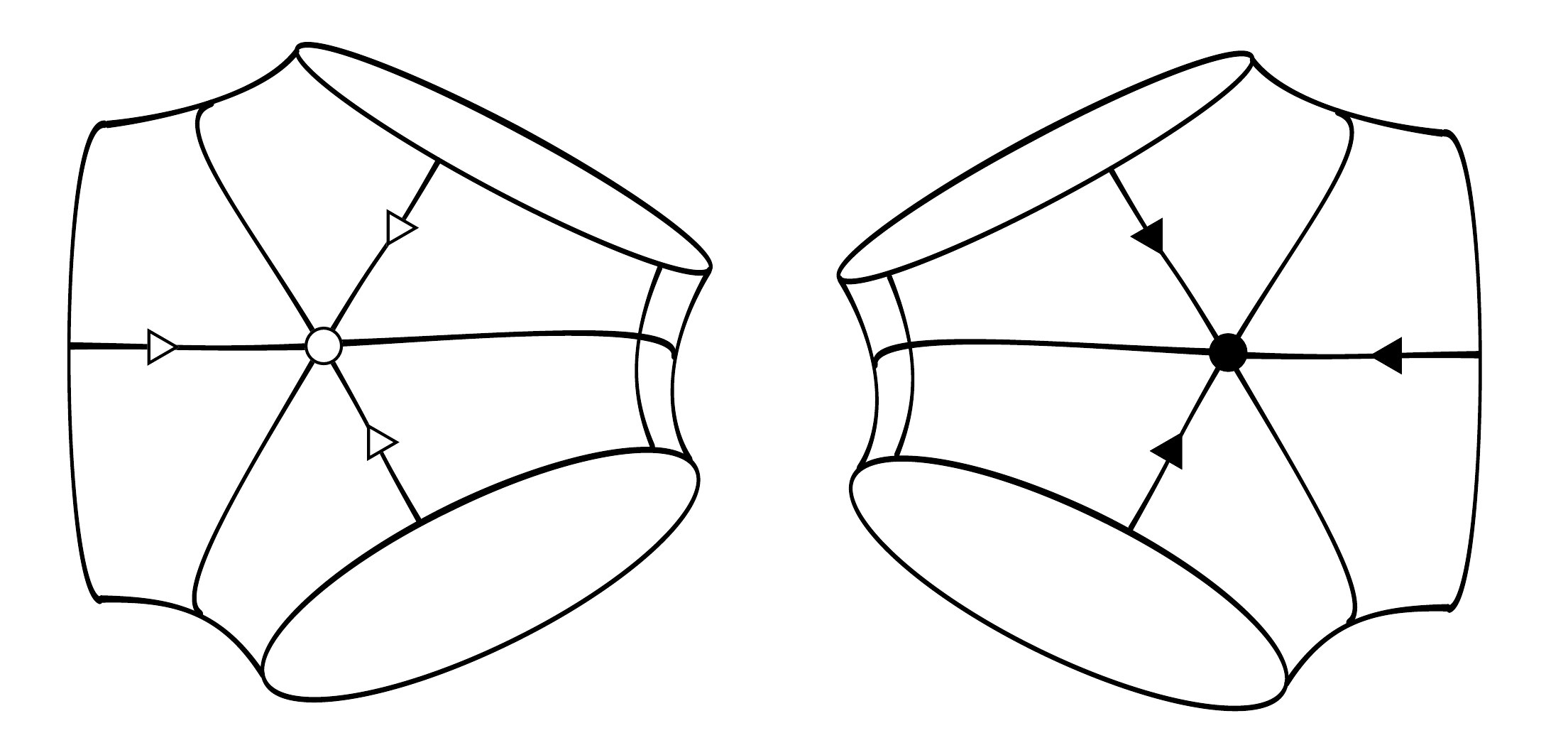,width=10.0cm,angle=0} }}
\vspace{-30pt}
\end{center}
\caption{The first segments of arcs on $P$ (front on left, back on right)} \label{fig:arcs}
}
\end{figure}

One immediate possibility is that, after the initial segment, the arc takes one of the remaining arcs to return to a boundary curve. This gives an arc of total length $1$. Up to homotopy, there are exactly three of these so we focus on arcs with length strictly greater than $1$. 

Consider an arc of length $L\geq 2$. After the initial arc, the arc necessarily follows "straight", otherwise the representative would have to make a left or a right turn at an angle of $\frac{\pi}{2}$, and hence wouldn't be geodesic. The arc is now comprised of a sequence of segments of length $1$ that go between the two singular points. Note that it alternates between the front singular point (in white in the figure) and the back one (in black in the figure). As we only consider geodesics, there is no backtracking, and so at each singular point, there are two choices (left or right). In particular, every oriented arc is determined by its initial and then a sequence of left and right moves. When we are done, there is a unique segment of length $\frac{1}{2}$ that finishes the arc geodesically. As we require length $\frac{1}{2}$ for the first segment, followed by a forced segment of length $1$, and an forced ending segment of length $\frac{1}{2}$, we have $2^{L-2}$ choices for the left/right paths. All in all, this means that there are 
$$
6\cdot 2^{L-2}
$$
oriented segments of length exactly $\geq 2$. Now each unoriented arc is counted exactly twice, and so we have 
$$
3\cdot 2^{L-2}
$$
unoriented arcs of length $L$. If we count all of those up to length $L$, we get 
$$
3 + 3 \sum_{l=2}^{L} 2^{l-2} = 3 + 3 \frac{1-2^{L-1}}{1-2}=3 \cdot 2^{L-1}.
$$
If we set $m=3 \cdot 2^{L-1}$, this gives a collection of $m$ arcs whose pairwise intersection never exceeds $L^2$. Restated in terms of $m$, this gives a collection of $m$ arcs with total crossing number at most
$$
\frac{m(m+1)}{2} L^2. 
$$
We have thus proved the following:

\begin{proposition}\label{prop:pantsarcs}
There exists a collection $\mathcal{A}$ of $m$ arcs on a pair of pants which satisfies
$$
\Cr(\mathcal{A}) \leq \frac{m(m+1)}{2} \left(\log_2\left(\frac{2m}{3}\right)\right)^2
$$
\end{proposition}

For what follows, it will be useful to consider collections of arcs whose endpoints both lie on a given boundary curve. For that we slightly modify the above construction. 

We highlight the differences and then give estimates:
\begin{enumerate}
\item We only have $2$ choices for the initial segment.\\
\item As we require both endpoints to be on the same cuff, for the final segment, depending in which direction we arrive on the singlular point, it may not be possible to end our arc by taking the path of length $\frac{1}{2}$. This is because this may not correspond to a geodesic. More precisely, suppose we want to finish on either $f_1$ or $b_1$, if we arrive on the segment facing it, we can finish our arc there. If not, we are required to add another segment of length $1$ which this time will be opposite the correct segment.\\
\end{enumerate}
In short, we proceed as above, and we obtain $2\cdot 2^{L-3}$ paths of length $L-\frac{3}{2}$ which we then complete into paths of length $L-1$ (if we are in front of the correspond boundary curve) or of length $L$. As we have counted orientation, we divide by $2$ to obtain a collection of $2^{L-3}=m$ arcs of length either $L-1$ or $L$. As before, their total crossing is at most 
$$
\frac{m(m+1)}{2} L^2= \frac{m(m+1)}{2} \left(\log_2\left(8m\right)\right)^2.
$$

\subsubsection{Curves}
We remain on $P$ but this time we count closed curves. We begin in a singular point and observe that a closed curve can be described by a sequence of segments of length $1$ which join the two singular points. As the curve is closed, there are an even number of them. 

Now we count them: for the initial segment, as before, we label $\{F_1,F_2,F_3\}$ the three segments leaving from the front singular point and $\{B_1,B_2,B_3\}$. Once this choice is made, a closed curve consists as before in a sequence of left/right moves. The curve can  end after an even number of segments if the first and the last segment aren't opposite (otherwise, viewed cyclically, the curve backtracks). To avoid this, we proceed as follows and we include counting:

\begin{enumerate}
\item We have $6$ choices for the initial segment.\\
\item For $k\geq 2$, we then take an even number $2L-2$ of segments consisting in left/right turns. There are $2^{2k-2}$ choices.\\
\item We are now on the opposite side of where we started. At least one of two possible choices of left/right (possibly both) allows us to close the curve up properly. 
\end{enumerate}
This results in a collection of $6\cdot 2^{2k-2}$ curves of length $L=2k$. 

We now account for overcounting: we begin by dividing by $2$ to account for both possible orientations, as well as by $L$ to account for our choice of "initial" segment. This gives a total of 
$$
m = 3 \frac{2^{L-2}}{L}
$$
of length $L$. Using $2^{n/2} \leq \frac{2^n}{n}$ for all $n\geq 4$, we obtain 
$$
2^{L/2}\leq \frac{4m}{3}
$$
and so 
$$
L \leq 2\cdot \log_2\left( \frac{4m}{3}\right).
$$
As before, any two of our curves cross at most $L^2$ times, and so we have a collection of $m$ curves which cross at most
$$
\frac{m(m+1)}{2} L^2 \leq 2 m(m+1) \left( \log_2\left( \frac{4m}{3}\right)\right)^2
$$
times.

As before, we summarize this construction in the following statement:

\begin{proposition}\label{prop:pantcurves}
There exists a collection $\mathcal{C}$ of $m$ homotopically distinct closed curves on a pair of pants which satisfies
$$
\Cr(\mathcal{C}) \leq 2 m(m+1) \left( \log_2\left( \frac{4m}{3}\right)\right)^2.
$$
\end{proposition}

\begin{figure}[h]
%\ShowGrid
{\color{linkblue}
\leavevmode \SetLabels
\L(0.255*.41) $F_1$\\
\L(0.366*.47) $F_2$\\
\L(0.29*.695) $F_3$\\
\L(0.715*.4) $B_1$\\
\L(0.61*.45) $B_2$\\
\L(0.665*.65) $B_3$\\
\endSetLabels
\begin{center}
\AffixLabels{\centerline{\epsfig{file =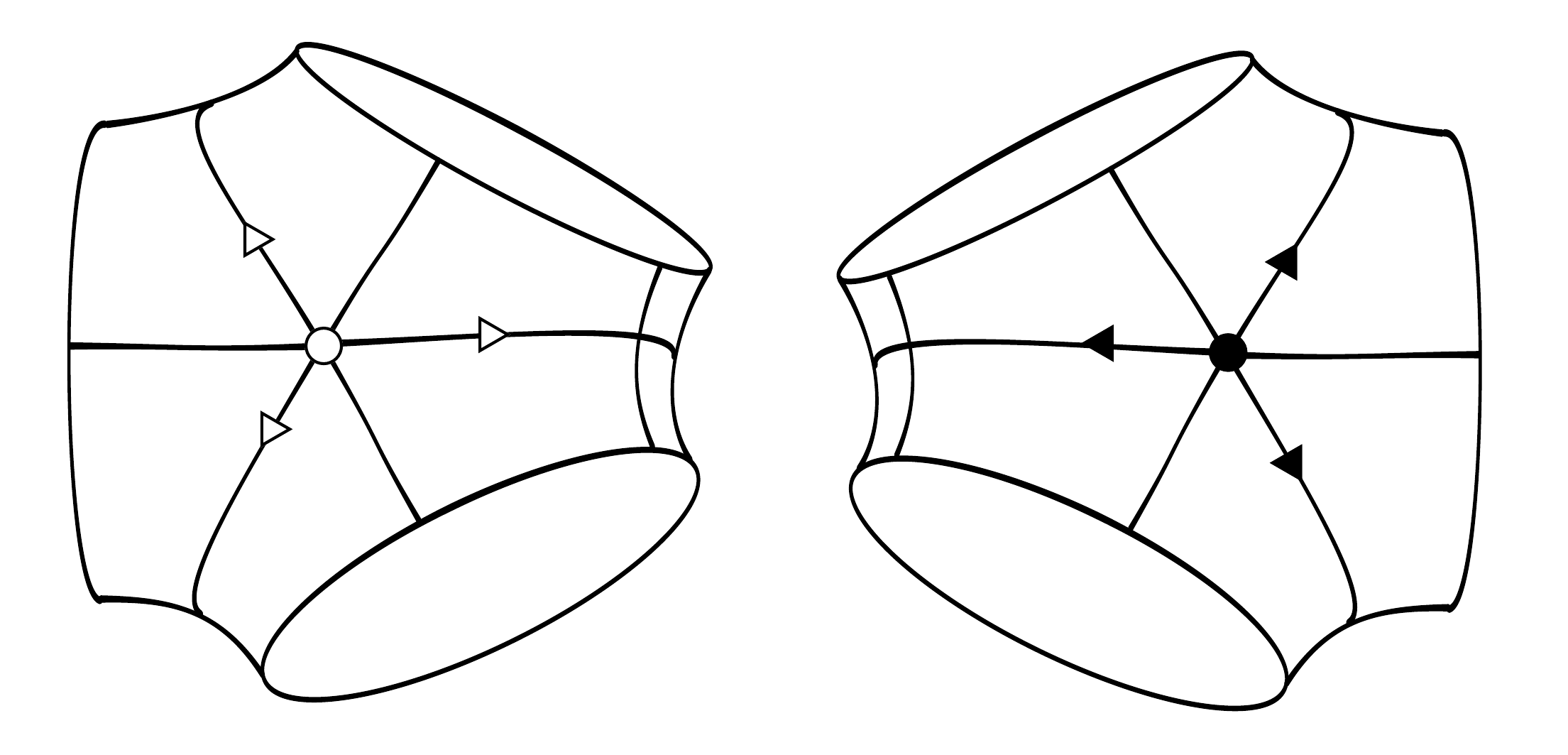,width=10.0cm,angle=0} }}
\vspace{-30pt}
\end{center}
\caption{The first segments of arcs on $P$ (front on left, back on right)} \label{fig:arcs}
}
\end{figure}
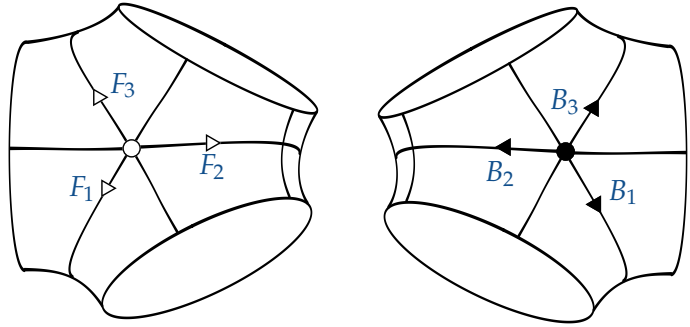

\subsection{Counting on a general surface and final estimates}

We now use the above estimates to obtain our general upper bounds. Consider a surface $\Sigma$ of Euler characteristic $\chi$. The idea is to split the curves or arcs that live on the individual pairs of pants. That way the arcs or curves only intersect a fraction of the others. 

We use this strategy in two distinct situations for our estimates, when $\Sigma$ is an $n\geq 3$ punctured sphere for arcs and when $\Sigma$ is a general surface for curves. 

\subsubsection{Arcs on a punctured sphere}
This case allows us to compare our estimates to those of Pach, Tardos and Toth \cite{pach-tardos-toth}. 

We split the sphere into $|\chi|=n-2$ pairs of pants such that each pair of pants has at least one of the boundary curves of $\Sigma$ as one of its cuffs (see Figure \ref{fig:sphere}). 

\begin{figure}[h]
%\ShowGrid
{\color{linkblue}
\leavevmode \SetLabels
\endSetLabels
\begin{center}
\AffixLabels{\centerline{\epsfig{file =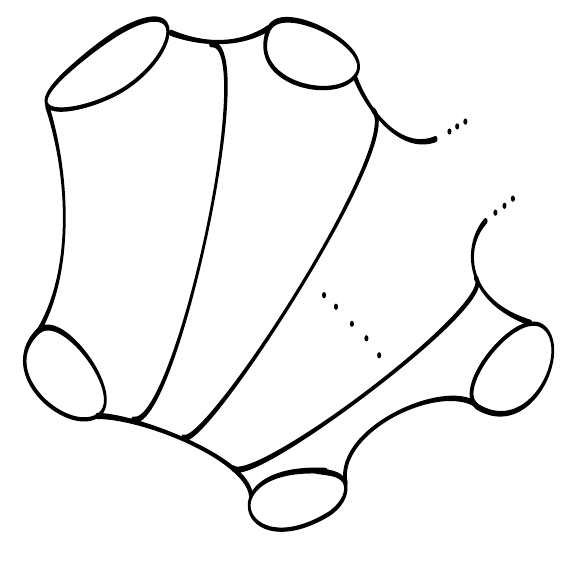,width=5.0cm,angle=0} }}
\vspace{-30pt}
\end{center}
\caption{A pants decomposition of a sphere with $n$ boundary curves} \label{fig:sphere}
}
\end{figure}

On each pair of pants, we use the above construction with arcs which end on the corresponding boundary curve of $\Sigma$ to obtain, for $L\geq 2$, collection of 

$$
m'= 2^{L-3}
$$
arcs of length either $L-1$ or $L$ on each pair of pants. The total number of arcs is $m=m'(n-2)$ and each one intersects at most $m' = \frac{m}{n-2}$ arcs. Using the fact that $L\leq \log_2(8 m') = \log_2\left(\frac{8m}{n-2}\right)$, this gives a collection of $m$ arcs with total crossing at most 
$$
\frac{m^2}{n-2} L^2 \leq \frac{m^2}{n-2} \left(\log_2\left(\frac{8m}{n-2}\right)\right)^2.
$$
This is the statement of Theorem \ref{thm:graphcrossingexample}. Using the same method, one can obtain estimates for graph drawings on more general surfaces.

\subsubsection{The case of curves}
Here we break the surface into $|\chi|$ pairs of pants and construct curves which lie on each pair of pants. 

On each pair of pants we construct $m'$ curves of length at most 
$$
L \leq 2\cdot \log_2\left( \frac{4m'}{3}\right).
$$
This gives a total of $m= |\chi| m'$ curves and each curve intersects at most $m'$ other curves at most $L^2$ times. This gives a total of at most
$$
\frac{m^2}{|\chi|}\left( \log_2\left( \frac{4m}{3|\chi|}\right)\right)^2.
$$
crossings, proving Theorem \ref{thm:curvesexample}. 

\bibliographystyle{hplain} % Choose a bibliography style
\bibliography{CountingCurves} % Specify the name of your .bib file (without the .bib extension)

{\it Addresses:}\\
Department of Mathematics, University of Fribourg, Switzerland. \\
Universit\'e Gustave Eiffel, CNRS, LIGM, Marne-la-Vall\'ee, France. \\

{\it Emails:}\\
hugo.parlier@unifr.ch\\
alfredo.hubard@univ-eiffel.fr

\end{document}